\documentclass[letterpaper,oneside, 10pt]{amsart}

\usepackage{amssymb,amsmath,latexsym,amsthm}
\usepackage[english]{babel}
\usepackage[hidelinks]{hyperref}
\usepackage{thmtools}
\usepackage{thm-restate}
\usepackage[vcentermath]{youngtab}

\declaretheorem[name=Theorem,numberwithin=section]{thm}
\newtheorem{theorem}[thm]{Theorem}
\newtheorem{lemma}[thm]{Lemma}
\newtheorem{corollary}[thm]{Corollary}
\newtheorem{prop}[thm]{Proposition}

\theoremstyle{definition}

\newtheorem{definition}[thm]{Definition}
\newtheorem{example}[thm]{Example}

\numberwithin{equation}{section}

\newcommand{\Fp}{\mathbb{F}_p}

\newcommand{\Zp}{\mathbb{Z}_p}

\newcommand{\Zptimes}{\mathbb{Z}_p^\times}

\DeclareMathOperator{\ordp}{ord}

\newcommand{\floor}[1]{\left \lfloor #1 \right \rfloor}

\begin{document}

\date{\today}
\title{A determinant of the Artin-Hasse exponential coefficients}
\author{Matthew Schmidt}
\email{mwschmid@buffalo.edu}

\begin{abstract}
We present a new determinant identity involving the coefficients of the Artin-Hasse exponential. In particular, if $E(x) = \exp(\sum_{k=0}^\infty \frac{x^{p^k}}{p^k}) = \sum_{n=0}^\infty u_nx^n$ is the Artin-Hasse exponential, we give, for any $\ell\geq 1$, a closed-form formula for the determinant $|u_{pi-j}|_{1\leq i,j\leq \ell}$ and show it is a $p$-adic unit.
\end{abstract}

\maketitle 

\numberwithin{theorem}{section}

\section{Introduction}

Define the Artin-Hasse exponential:
\begin{align}\label{line:AH}
	E(x) = \exp(\sum_{k=0}^\infty \frac{x^{p^k}}{p^k}) = \sum_{n=0}^\infty u_nx^n\in (\Zp\cap\mathbb{Q})[[x]].
\end{align}
This exponential, and variations of it, play an essential role in the Dwork theory of exponential sums and the explicit computation of $L$-functions over finite fields. Its coefficients also have a rich combinatorial interpretation which will be the focus of this paper. Let $h_n$ be the number of $p$-elements in the symmetric group $S_n$. It can be shown (Lemma~2.11 in \cite{Kracht}) that
\[
	u_n = \frac{h_n}{n!}. 
\]

To date, there are few explicit results involving these coefficients. 
One can use the power series definition from Equation~(\ref{line:AH}) and compute an explicit expansion of $u_n$ (see Equation~(2.9) in \cite{Kracht}).  It is also known that the coefficients satisfy a recursive formula (Lemma~1 in \cite{KS}):
\[
	u_n = \frac{1}{n}\sum_{i=0}^\infty u_{n-p^i},
\]
where we initially take $u_0=1$.  More generally,  Conrad has shown (see the proof of Theorem~2.10 in \cite{KConrad}) that the Artin-Hasse exponental $\bmod\ p$ is not rational over $\Fp$, that is,  $(E(x)\bmod\ p)\not\in \Fp(x)$.

Other properties about the coefficients are less resolved, and the main mystery is how the coefficients behave $p$-adically.
Conrad has  conjectured (Section~5 of \cite{KConrad}) that they are uniformly distributed in $\Zp$ with respect to the Haar measure.  Thakur \cite{Thakur} has also asked whether the Artin-Hasse exponential $\bmod\ p$ is transcendental over $\Fp(x)$.  In terms of the coefficients, Thakur's question is whether or not the set of sequences of coefficients $\bmod\ p$,  
\begin{align*}\label{line:pkernel}
	\{({u}_{p^in+j}\bmod\ p)_{n\geq 0} | i\geq 0\textrm{ and } 0\leq j<p^i\},
\end{align*}
called the $p$-kernel of the sequence $(u_n\bmod\ p)_{n\geq 0}$, is finite.  See Section~6.6 in \cite{Allouche} for further information on $p$-kernels and Section~12.2 for automatic sequences over finite fields in general.

In this paper, we present a new identity involving the Artin-Hasse exponential coefficients.  Before we state our result,  let us motivate it.  Define the Cartier operator on $\Zp[[x]]$ by:
\begin{align*}
	U_p: \Zp[[x]] &\to \Zp[[x]]\\
	\sum_{n=0}^\infty a_nx^n &\mapsto \sum_{n=0}^\infty a_{pn}x^n,
\end{align*}
and construct an operator:
\begin{align*}
	\phi: \Zp[[x]]\ &\to \Zp[[x]] \\
		f(x) &\mapsto U_p(E(x)\cdot f(x)),
\end{align*}
where the action inside the $U_p$ is just multiplication.  The Cartier operator appears in both the Dwork theory of exponential sums (Chapter 5, \S 3 of \cite{Koblitz}) and Christol's theory of automatic sequences over finite fields  (Section~12.2 of \cite{Allouche}).  The matrix of this $\phi$ operator with respect to the basis $\{1, x, x^2, \cdots\}$ is then:
\[	
	\begin{bmatrix}
		1 & u_p & u_{2p} & u_{3p} &  \\
		0 & u_{p-1} & u_{2p-1} & u_{3p-1} & \cdots \\
		0 & u_{p-2} & u_{2p-2} & u_{3p-2} &  \\
		 & & \vdots & 
	\end{bmatrix}.
\] 
Our main result is the explicit computation of the leading principal minors of the transpose of this matrix.   Given the general difficulty in dealing with the coefficients $u_n$, especially when $n>p-1$,  it is surprising that a closed-form formula for these determinants exists. Explicitly, 
\begin{theorem}\label{thm:main}
For any $1\leq \ell$,
\[
	|u_{pi-j}|_{1\leq i,j\leq \ell} = \prod_{k=1}^\ell\frac{k!p^k}{(pk)!}\in\Zptimes,
\]
where we take $u_n=0$ if $n<0$.
\end{theorem}
To prove Theorem~\ref{thm:main}, we reduce the determinant $|u_{pi-j}|_{1\leq i,j\leq \ell}$ to one involving only binomials. Then, using a result of Tonne \cite{Tonne}, we are able to explicitly compute the determinant by counting certain types of Young tableaux. 

This paper has three main sections. The first section is essentially independent from the rest of the paper,  and in it we count the Young tableaux needed in order to get a closed-form formula from Tonne's theorem.  In the second section, we compute an expansion of $h_n$, Lemma~\ref{lemma:h_n_expans}, which we use in the third section after exploiting a determinant trick, Lemma~\ref{lemma:deter}, to finally prove Theorem~\ref{thm:main}.


\section{Young Tableaux}\label{section:young}

We say a Young tableaux $T$ has shape $(\lambda_1, \lambda_2, \cdots, \lambda_n)$ if it has $n$ rows and the $k$th row has $\lambda_k$ blocks in it. All of our Young tableaux will contain integers in each block and we will use the notation $(T)_{i,j}$ to refer to the integer in the $j$th block in the $i$th row.

\begin{example}
Let $T$ be the Young tableaux 
\[
	\young(123,45,6).
\]
Then $T$ has shape $(3,2,1)$ and, for example, $(T)_{1,2}=2$ and $(T)_{2,2}=5$.
\end{example}

\begin{definition}
Let $T_n$ be the set of all Young tableaux $T$ such that:
\begin{enumerate}
	\item $T$ has shape $(n, n-1,\cdots, 1)$.
	\item For each block, $i,j$: \[
			1\leq (T)_{i,j} \leq p + \sum_{k=1}^{i-1}((T)_{k,j+1}-(T)_{k,j}).
		\]
\end{enumerate}
\end{definition}

For two tableaux in $T_{n-1}$ and one $u\in\mathbb{Z}_{\geq 1}$, define a ``gluing'' operation $G(S,T,u)$, resulting in a new tableaux with shape $(n, n-1, \cdots, 1)$, in the following way:
\begin{enumerate}
	\item Take the first column of $S$ and glue it to the left side of $T$ to generate a new young diagram with shape $(n+1, n, \cdots, 2)$.
	\item Glue $u$ to the bottom of the first column of this new diagram to create a diagram with shape $(n+1, n, \cdots, 1)$. 
\end{enumerate}

\begin{example} Suppose $S$ and $T$ are both shape $(3,2,1)$:
\[
	G\left (\young(abc,de,f), \young(ghi,jk,l), \young(m)\right ) =  \young(aghi,djk,fl,m).
\]
\end{example}

\begin{definition}
Suppose that $(S,T,u)\in T_{n-1}\times T_{n-1}\times \mathbb{Z}_{\geq 1}$ such that
	\[
		1 \leq u \leq p + \sum_{k=1}^{n-1} ((T)_{k,1}-(S)_{k,1}).
	\]
We call such a triple $n$-admissible. 
\end{definition}

\begin{lemma}\label{lemma:num_admiss}
For any $n\geq 1$, the number of $n$-admissible triples is $|T_{n-1}|^2*p$.
\end{lemma}

\begin{proof}
Summing over each possibility of $u$ for each pair of $S,T$ yields:
\begin{align}\label{line_tnp}
		\sum_{S,T\in T_{n-1}} p + \sum_{k=1}^{n-1} ((T)_{k,1}-(S)_{k,1}) = |T_{n-1}|^2*p +\sum_{S,T\in T_n} \sum_{k=1}^{n-1} ((T)_{k,1}-(S)_{k,1}).
\end{align}
But then $\sum_{S,T\in T_{n-1}} \sum_{k=1}^{n-1} ((T)_{k,1}-(S)_{k,1})=0$ by symmetry and so Equation~(\ref{line_tnp}) is simply $|T_{n-1}|^2*p$.
\end{proof}

Define two maps: $F: T_n \mapsto T_{n-1}$ and $L: T_n\mapsto T_{n-1}$ given by truncating the first block on each row and the last block on each row, respectively. 
\begin{example}
\begin{align*}
	F\left ( \young(abc,de,f)\right )= \young(bc,e)\ \ \ \ \ \ 
	L\left ( \young(abc,de,f)\right )= \young(ab,d)
\end{align*}
\end{example}

Clearly if $(S, T, u)$ is an $n$-admissible triple with $F(S) = L(T)$, then $G(S, T, u)\in T_n$.
\begin{prop}\label{prop:admiss_bij}
The map $G(S, T,u)$ is a bijection between $n$-admissible triples $(S, T, u)$ with $F(S) = L(T)$ and tableaux in $T_n$.
\end{prop}
\begin{proof}
If $P\in T_n$, then $G(L(P), F(P), (P)_{n,1}) = P$, so $G$ is surjective.

Suppose that $(S,T,u)$ and $(S', T', u')$ are $n$-admissible triples with $F(S) = L(T)$ and $F(S')=L(T')$ such that $G(S,T,u)=G(S',T',u')$. Immediately we see that $u=u'$ and $T=T'$ by the construction of $G$. For the same reason, the first columns of $S$ and $S'$ also coincide. But then $F(S')=L(T')=L(T)=F(S)$, and so the rest of the tableaux of $S$ and $S'$ coincide. Thus $G$ is injective.
\end{proof}

\begin{corollary}\label{Tn_recursive}
For any $n\geq 3$:
	\[
		|T_n| = \frac{|T_{n-1}|^2}{|T_{n-2}|}\cdot p.
	\]
\end{corollary}
\begin{proof}
By the bijection in Proposition~\ref{prop:admiss_bij}., each $P\in T_n$ corresponds to an $n$-admissible triple $(S, T, u)$ with $F(S)=L(T)$.  By Lemma~\ref{prop:admiss_bij}, the total number of $n$-admissible triples is $|T_{n-1}|^2*p$, and exactly $1/|T_{n-2}|$ many of these have the condition that $F(S)=L(T)$ since $L(T), F(S)\in T_{n-2}$.
\end{proof}

\begin{corollary}\label{corr:tn}
	For any $n\geq 1$, $|T_n|=p^{\sum_{k=1}^n} k$.
\end{corollary}
\begin{proof}
If $n=1$, then clearly $|T_1|=p$. For $n=2$, suppose the following is a tableaux in $T_2$:
	\[
	\young(ab,c)
\]
The only condition on $a$ and $b$ is that $1\leq a,b\leq p$, and for each $(a,b)\in [1,p]^2$, there are $p+(b-a)$ choices for $c$. Thus in total there are:
\begin{align*}
	\sum_{(a,b)\in [1,p]^2} (p + (b-a)) &= \sum_{(a,b)\in [1,p]^2} p + \sum_{(a,b)\in [1,p]^2}(b-a)
\end{align*}
 tableaux in $T_2$. But then $\sum_{(a,b)\in [1,p]^2} p =p^3$ and by symmetry $\sum_{(a,b)\in [1,p]^2}(b-a)=0$. Hence $|T_1|=p$ and $|T_2|=p^3$. 
For general $n$ we proceed by induction:
	\[
		|T_n| = \frac{|T_{n-1}|^2}{|T_{n-2}|}\cdot p=p^{2\sum_{k=1}^{n-1}k - \sum_{k=1}^{n-2}k + 1} = p^{n + \sum_{k=1}^{n-1}k} = p^{\sum_{k=1}^n k}.
	\]
\end{proof}

\section{The Artin-Hasse Coefficients}

\begin{definition}
	Let $I_n$ be the set of all cycle types having $p$-power order in $S_n$. We do not include fixed points (represented by 1's) in our cycle types. If $\sigma\in I_n$, denote by $|\sigma|$ the length of the cycle type, again not counting fixed points.
Let $I^k$ be the set of all cycle types $t$ with $p$-power order and cycle length $pk$. For each cycle type $t$, let $C_t^n$ be the number of $p$-cycles with cycle type $t$ in $S_n$. 
\end{definition}

\begin{example}
Suppose $p=3$. Then in $S_{10}$:
\[
	I_{10} = \{(3), (3,3), (3,3,3), (9)\},
\] 
and if $t=(3,3,3)$, then $|t|=3+3+3=9$. Further,
\begin{align*}
	I^1 &= \{(3)\}, \\
	I^2 &= \{(3,3)\}, \\
	I^3 &= \{(3,3,3), (9)\}.
\end{align*}
\end{example}

By definition of $h_n$:
\begin{align*}
	h_n = 1+ \sum_{\substack{t\in I_n\\ t\neq 1}} C_t^n.
\end{align*}

\begin{definition}
Let $n\in \mathbb{Z}_{\geq 2}$. Define a map:
	\begin{align*}
		r_n:S_{n-1}\times \{1,\cdots, n-1\}\mapsto S_n,
	\end{align*}
where $r_n(\sigma, j)$ is given by replacing the element $j$ in $\sigma$ by $n$.
\end{definition}

\begin{lemma}\label{lemma:C_tn_recurs}
If $t\neq 1$ and $n > |t|$, then:
\[
	C_t^n = \frac{n}{n-|t|}C_t^{n-1}.
\]
\end{lemma}
\begin{proof}
Consider the multiset $K$:
\[
	K = \bigcup_{\sigma\in C_t^{n-1}}\{ r_n(\sigma, 1), r_n(\sigma,2), \cdots, r_n(\sigma, n-1), \sigma\}.
\]
Clearly $C_t^n\subset K$, so we will count and see with what multiplicity each element of $C_t^n$ appears in $K$.

If $j$ does not appear in some permutation $\sigma\in C_t^{n-1}$, then $r_n(\sigma, j)=\sigma$. So for each $\sigma\in C_{t}^{n-1}$, because there are $(n-1)-|t|$ numbers not appearing in $\sigma$, $\sigma$ itself appears in the set $K$ with multiplicity $(n-1)-|t|+1=n-|t|$. 
On the other hand, suppose $\sigma\in C_t^n-C_t^{n-1}$ so $\sigma$ has an $n$ in its cycle somewhere. By replacing $n$ with all the numbers not appearing in $\sigma$, we can see there are $n-|t|$ many $\tau\in C_t^{n-1}$ such that $r_n(\tau, j)=\sigma$, some $j$. 

Therefore since every element in $K$ has multiplicity $n-|t|$ and $K$ has cardinality $nC_t^{n-1}$, the claim follows.
\end{proof}

\begin{corollary}\label{cor:C_tn_binom}
If $t\neq 1$, then:
\begin{align*}
	C_t^n = \binom{n}{|t|}C_t^{|t|}.
\end{align*}
\end{corollary}
\begin{proof}
Inductively applying Lemma~\ref{lemma:C_tn_recurs} yields:
\begin{align*}
	C_t^n &= \frac{n}{n-|t|}C_t^{n-1} =  \frac{n}{n-|t|}\cdot  \frac{n-1}{n-1-|t|} C_t^{n-2}\\
		&=\frac{n}{n-|t|}\cdot  \frac{n-1}{n-1-|t|}\cdot \frac{n-2}{n-2-|t|}\cdots \frac{|t|+1}{1}C_t^{|t|}.
\end{align*}
\end{proof}

\begin{lemma}\label{lemma:h_n_expans}
For $n\geq 1$,
\begin{align*}
	h_n = \sum_{k=0}^{\floor{n/p}} \binom{n}{kp} C_k,
\end{align*}	
where $C_k = \sum_{t\in I^k}C_t^{kp}$.
\end{lemma}
\begin{proof}
The lemma follows from Corollary~\ref{cor:C_tn_binom}, and because:
\begin{align*}	
	h_n = 1+\sum_{\substack{t\in I_n\\t\neq 1}}C_t^n = \sum_{k=0}^{\floor{n/p}}\sum_{t\in I^k} C_t^n =  \sum_{k=0}^{\floor{n/p}}\sum_{t\in I^k}\binom{n}{kp}C_t^{kp} =  \sum_{k=0}^{\floor{n/p}}\binom{n}{kp}\cdot (\sum_{t\in I^k}C_t^{kp}).
\end{align*}
\end{proof}

\section{The Determinant}\label{section:ah}

In this section we will piece together the results of the previous two sections and prove our main theorem. 
We start with a technical lemma which is key to proving Proposition~\ref{prop:CD_det}.

\begin{lemma}\label{lemma:deter}
For any formal variables $E_{j,k}$ and $X_{i,k}$, 
\[
	|  \sum_{k=1}^{\ell} E_{j, i-k+1} X_{i,k} |_{1\leq i,j\leq \ell} = |E_{j,k}|_{1\leq j,k\leq \ell} \cdot \prod_{i=1}^\ell X_{i,1}.
\]
\end{lemma}
\begin{proof}
Consider the matrix $M$:
\begin{align*}
	M = |E_{j,i}|_{1\leq i,j\leq \ell}\cdot |X_{i, i-j+1}|_{1\leq i,j\leq\ell},
\end{align*}
taking $X_{i,k}=0$ if $k<1$. The entry $(M)_{i,j}$ is given by:
\[
	(M)_{i,j} = \sum_{v=1}^\ell X_{i, i-v+1}E_{j, v} = \sum_{v=1}^{i} X_{i, i-v+1}E_{j, v}.
\]
The claim then follows since $ |X_{i, i-j+1}|_{1\leq i,j\leq\ell}$ is lower triangular and equals\\ $\prod_{i=1}^\ell X_{i,1}$.
\end{proof}

\begin{prop}\label{prop:CD_det}
Let $i,j\geq 1$ and $\alpha_j,\beta_i\in\mathbb{Q}_{>0}$. For any $\ell\geq 1$,
\begin{align*}
	|\alpha_j\beta_iu_{pi-j}|_{1\leq i,j\leq \ell} =  |\frac{\alpha_j\beta_i}{(pk-j)!}|_{1\leq j,k\leq \ell}.
\end{align*}
\end{prop}
\begin{proof}
Let $i,j\geq 1$ with $pi-j\geq 0$. By Lemma~\ref{lemma:h_n_expans}:
\begin{align*}
	h_{pi-j} = \sum_{k=1}^{i-\floor{j/p}}\binom{ip-j}{kp} C_k.
\end{align*}
Hence:
\begin{align}\label{line:CjDih}
	\frac{\alpha_j\beta_i}{(pi-j)!}h_{pi-j} &= \sum_{k=0}^{i-\floor{j/p}}\left (\frac{\alpha_j}{(p(i-k)-j)!}\right )\left (\beta_i\frac{C_k}{(pk)!}\right).
\end{align}

Thus if $X_{i,k} = \beta_i\frac{C_k}{(pk)!}$ and
\begin{align*}
E_{j, k} = \begin{cases}
		\frac{\alpha_j}{(pk-j)!} & k\geq  \floor{j/p}\\
		0 & k< \floor{j/p}
		\end{cases},
\end{align*} 
then Equation~(\ref{line:CjDih}) becomes:
\begin{align}
\frac{\alpha_j\beta_i}{(pi-j)!}h_{pi-j} &= \sum_{k=0}^{\ell}E_{j,i-k}X_{i,k}.
\end{align}
So by Lemma~\ref{lemma:deter},
\begin{align*}
	|\alpha_j\beta_iu_{pi-j}|_{1\leq i,j\leq \ell} &=|\frac{\alpha_j\beta_i}{(pi-j)!}h_{pi-j}|_{1\leq i,j\leq \ell} =  |E_{j,k}|_{1\leq j,k\leq \ell} \cdot \prod_{i=1}^\ell X_{i,0} \\
	&=  |\frac{\alpha_j\beta_k}{(pk-j)!}|_{1\leq j,k\leq \ell}.
\end{align*}
\end{proof}

\begin{corollary}\label{corr:binom}
For any $\ell\geq 1$,
\[
	|\binom{pi}{j}h_{pi-j}|_{1\leq i,j\leq \ell} = |\binom{pi}{j}|_{1\leq i,j\leq \ell}.
\]
\end{corollary}
\begin{proof}
Taking $\alpha_j=1/j!$ and $\beta_i = (pi)!$ in Lemma~\ref{prop:CD_det} yields:
\begin{align*}
	|\frac{(pi)!}{j!}u_{pi-j}|_{1\leq i,j\leq \ell} &= 
	|\binom{pi}{j}h_{pi-j}|_{1\leq i,j\leq \ell} =  |\frac{(pi)!}{j!(pi-j)!}|_{1\leq i,j\leq \ell} = |\binom{pi}{j}|_{1\leq i,j\leq \ell}.
\end{align*}
\end{proof}

\begin{theorem}\label{thm:hpij_det}
For any $\ell\geq 1$,
\[
	|\binom{pi}{j}h_{pi-j}|_{1\leq i,j\leq \ell} = p^{\sum_{k=1}^\ell k}.
\]
\end{theorem}
\begin{proof}
By Corollary~\ref{corr:binom} and Theorem 1 in \cite{Tonne},
\begin{align*}
|\binom{pi}{j}h_{pi-j}|_{1\leq i,j\leq \ell} = |\binom{pi}{j}|_{1\leq i,j\leq \ell} = p^\ell\cdot |T_{n-1}|,
\end{align*}
and the theorem follows by Corollary~\ref{corr:tn}.
\end{proof}

\begin{corollary}\label{corr:hass_nonzero}
For any $\ell\geq 1$,
\[
	|u_{pi-j}|_{1\leq i,j\leq \ell} = \prod_{k=1}^\ell\frac{k!p^k}{(pk)!}\in\Zptimes.
\]
\end{corollary}
\begin{proof}
Observe that 
\[
	\binom{pi}{j}h_{pi-j} = \frac{(pi)!}{j!(pi-j)!}h_{pi-j} = \frac{(pi)!}{j!}u_{pi-j},
\]
so that $u_{pi-j} = \frac{\binom{pi}{j}h_{pi-j}}{\frac{(pi)!}{j!}}$. Then, by Theorem~\ref{thm:hpij_det}, 
\[
	|u_{pi-j}|_{1\leq i,j\leq \ell} = \frac{|\binom{pi}{j}h_{pi-j}|_{1\leq i,j\leq \ell}}{\prod_{i=1}^\ell \frac{(pi)!}{i!}}=\frac{p^{\sum_{k=1}^\ell k}}{\prod_{i=1}^\ell \frac{(pi)!}{i!}}=\prod_{k=1}^\ell\frac{k!p^k}{(pk)!},
\]
and so since
\begin{align*}
\ordp_p\left( \frac{k!p^k}{(kp)!}\right)= \sum_{j=1}^\infty\floor{\frac{k}{p^j}}+k-\sum_{j=1}^\infty\floor{\frac{kp}{p^j}}=0,
\end{align*}
the claim follows.
\end{proof}


\end{document}